\documentclass[12pt]{amsart}
\usepackage[T1]{fontenc}
\usepackage{lmodern}
\usepackage{amsmath,amssymb,amsthm,mathtools}
\usepackage[margin=1in]{geometry}
\usepackage{microtype}
\usepackage[unicode,colorlinks=true,linkcolor=black,citecolor=black,urlcolor=blue]{hyperref}
\numberwithin{equation}{section}
\newtheorem{theorem}{Theorem}[section]
\newtheorem{lemma}[theorem]{Lemma}
\newtheorem{proposition}[theorem]{Proposition}
\newtheorem{corollary}[theorem]{Corollary}
\theoremstyle{remark}
\newtheorem{example}[theorem]{Example}
\newtheorem{remark}[theorem]{Remark}
\newcommand{\C}{\mathbb C}
\newcommand{\R}{\mathbb R}
\newcommand{\Nn}{\mathbb N_0}
\newcommand{\cA}{\mathcal A}
\newcommand{\cP}{\mathcal P}
\newcommand{\opRe}{\operatorname{Re}}
\newcommand{\intr}{\operatorname{int}}

\title{A Kallin's lemma for polynomials with exponents in a compact convex set and local approximation}
\author{Kieu Phuong Chi and Duong Thi Thuy Van}
\date{\today}

\address{Kieu Phuong Chi,  \newline
Department of Mathematics and Applications, Saigon  University, 273  An Duong Vuong, Ho Chi Minh City, Vietnam}
\email{\href{mailto:kieuphuongchi@sgu.edu.vn}{kieuphuongchi@sgu.edu.vn}}

\address{Duong Thi Thuy Van,  \newline
Faculty of Natural Science Education, Dong Nai University, Vietnam}
\email{\href{mailto:vandtt@dnpu.edu.vn}{vandtt@dnpu.edu.vn}}
\begin{document}

\begin{abstract}
Let $S\subset\R_+^n$ be compact and convex with $0\in S$, and let $\cP^S$ be the algebra spanned by monomials with exponents in $\bigcup_{m\ge0}(mS\cap\Nn^n)$. We define the associated polynomial hull and prove a Kallin lemma for the union of two sets convex with respect to this algebra. A weighted approximation lemma then shows that, under a sector or ray condition, two admissible monomials generate exactly the functions of the unrestricted algebra that are constant on a specified zero fibre. Applications include an exact density criterion for mixed monomials, a transfer principle for relative perturbations, and density and hull computations on stars of totally real discs.
\end{abstract}
\maketitle

\section{Introduction}

Let $D$ be a small closed disk centred at $0\in\C$. A classical problem asks when two functions $f,g\in C(D)$ generate $C(D)$ by uniform polynomial approximation. Wermer \cite{wermer1964} proved this for $f=z$ and $g\in C^1$ with $g_{\bar z}(0)\ne0$. Minsker \cite{minsker1976} obtained the exact condition $\gcd(m,n)=1$ for $z^m$ and $\bar z^{\,n}$, and de Paepe \cite{depaepe1986} established a perturbation result in the coprime case under differentiability hypotheses. Without coprimality, higher order terms matter: $z^2$ and $\bar z^2+z^3$ generate $C(D)$ \cite{depaepe1996}, whereas the image of $(z^2,\bar z^2+\bar z^3)$ has a nontrivial polynomial hull \cite{os2002}.

Many positive results use Kallin's lemma to join totally real discs obtained by lifting the problem under a power map; see \cite{depaepe2001,dieudepaepe2002,dieuchi2006,dpw2007,chi2023, chi2026}. In its classical form, the lemma asserts that two polynomially convex compact sets have polynomially convex union if a polynomial maps them into polynomially convex planar sets meeting only at $0$, with $0$ a boundary point of both sets, and the zero fibre on the union is polynomially convex \cite{kallin1966}.

We study the same approximation problem with restricted exponents. Given a compact convex set $S\subset\R_+^n$ containing $0$, let $\cP^S$ be the union of the spaces spanned by monomials with exponents in $mS\cap\Nn^n$. Such spaces occur in the pluripotential theory of convex bodies \cite{bbl2018,msss2024,mss2025,ChiVan2026}. Their union depends only on the cone $\Gamma:=\R_+S$. If this cone misses both coordinate axes away from $0$, neither coordinate monomial is admissible.

Our first result is a hull version of Kallin's lemma for $\cP^S$. If $X_1,X_2$ are $S$-polynomially convex, an admissible polynomial maps them into planar polynomially convex sets meeting only at the boundary point $0$, and its zero fibre on their union is $S$-polynomially convex, then their union is $S$-polynomially convex. Lemma~\ref{lem:kallinHull} also treats relative hulls in a specified ambient set. This distinction matters because restricted polynomials can identify entire coordinate axes, making the ambient hull unbounded even when uniform approximation on the original compact set holds.

We then identify a closed algebra generated by two admissible monomials. In two variables, choose $\beta\in\intr\Gamma\cap\Nn^2$, put $h:=f^{\beta_1}g^{\beta_2}$, and take $N$ large enough that $F:=h^Nf$ and $G:=h^Ng$ are admissible. Under a sector or ray condition on $h(K)$, Theorem~\ref{thm:core} gives
\[
 \overline{\C[F,G]}=\{u\in\overline{\C[f,g]}:u|_{h^{-1}(0)}\text{ is constant}\}.
\]
If $\Gamma$ misses both axes away from $0$, this is also the whole restricted algebra. The proof first recovers $h$ from $h^{N(\beta_1+\beta_2)+1}=F^{\beta_1}G^{\beta_2}$, then removes powers of $h$ by the weighted Kallin lemma, Lemma~\ref{lem:kallin}.

In two variables the extreme slopes $a_S,b_S$ determine the interior of $\Gamma$; boundary membership must also be specified to determine the full algebra. For mixed monomials with opposite angular degrees $p>0$ and $-q<0$, Theorem~\ref{thm:sharp} gives density exactly when $\gcd(p,q)=1$ and $a_S<p/q<b_S$. Corollary~\ref{cor:perturb} identifies restricted and unrestricted algebras for continuous relative perturbations. Section~\ref{sec:examples} illustrates zero-fibre collapse and irrational cones, and gives an arithmetic criterion for two generators on stars of totally real discs. The opposite outcomes for the two cubic perturbations above persist in the restricted setting.

\section{Restricted polynomial hulls and Kallin's lemma}

Throughout, $S\subset\R_+^n$ is a compact convex set with $0\in S$, and $\Nn=\{0,1,2,\dots\}$. For an integer $m\ge0$ let $\cP^S_m$ be the span of the monomials $Z^\alpha$ with $\alpha\in mS\cap\Nn^n$, and let $\cP^S=\bigcup_m\cP^S_m$. Since $0\in S$ and $S$ is convex, $mS\subset m'S$ for $m\le m'$, the union is increasing, and $(mS)+(m'S)\subset(m+m')S$. Every point of the convex cone $\Gamma:=\R_+S$ lies in some integer dilate of $S$, so with
\[
 M:=\Gamma\cap\Nn^n
\]
we have
\begin{equation}\label{eq:semigroup}
 \cP^S=\operatorname{span}\{Z^\alpha:\alpha\in M\},
\end{equation}
and $\cP^S$ is an algebra because $M$ is an additive semigroup. We do not assume that $\Gamma$ is closed or rational polyhedral, nor that $M$ is finitely generated; for instance $S=\{(x,y):0\le x\le1,\ x^2\le y\le1\}$ gives a cone that contains the vertical axis but not the horizontal one. The identity \eqref{eq:semigroup} immediately gives the following observation.

\begin{proposition}\label{prop:cones}
If $S_1,S_2\subset\R^n_+$ are compact and convex, contain $0$, and have the same conical hull, then $\cP^{S_1}=\cP^{S_2}$.
\end{proposition}

In two variables Proposition~\ref{prop:cones} can be made completely explicit. For $S\ne\{0\}$ we put
\begin{equation}\label{eq:slopes}
 a_S:=\inf\Bigl\{\frac yx:(x,y)\in S\setminus\{0\}\Bigr\},\qquad
 b_S:=\sup\Bigl\{\frac yx:(x,y)\in S\setminus\{0\}\Bigr\},
\end{equation}
with the convention $y/0=+\infty$, so that $0\le a_S\le b_S\le\infty$. The directions of the convex cone $\Gamma$ form an interval with end slopes $a_S$ and $b_S$; hence $\Gamma$ has nonempty interior exactly when $a_S<b_S$, and then
\begin{equation}\label{eq:interior}
 \intr\Gamma=\{(x,y):x>0,\ a_Sx<y<b_Sx\}.
\end{equation}
Moreover $\Gamma$ meets the coordinate axes only at the origin as soon as $a_S>0$ and $b_S<\infty$. Thus, up to the question whether the two boundary rays belong to $\Gamma$, the algebra $\cP^S$ is determined by the two numbers $a_S,b_S$. When $S$ is the Newton polygon $\operatorname{conv}\{0,\alpha^{(1)},\dots,\alpha^{(k)}\}$ of a polynomial with nonzero constant term, $a_S$ and $b_S$ are the smallest and largest slopes of the vectors $\alpha^{(i)}$; for instance the Newton polygon of $1+Z^2W+ZW^2$ has $a_S=1/2$ and $b_S=2$, and both boundary rays belong to $\Gamma$. The interior-direction conditions below are expressed through $a_S$ and $b_S$. Boundary membership is retained explicitly when it matters. For a compact set $X\subset\C^n$ we write $P(X)$ and $\cP^S(X)$ for the closures in $C(X)$ of all polynomials and of $\cP^S$, respectively. For a compact Hausdorff space $K$ and $f,g\in C(K)$ we put
\[
 \cA_S(f,g;K):=\overline{\{P(f,g):P\in\cP^S\}},\qquad
 \cA_0(f,g;K):=\overline{\C[f,g]},
\]
the closures being taken in $C(K)$. Since $\|P(f,g)\|_K=\|P\|_X$ for $X:=(f,g)(K)$, the map $\varphi\mapsto\varphi\circ(f,g)$ is an isometric isomorphism of $\cP^S(X)$ onto $\cA_S(f,g;K)$ and of $P(X)$ onto $\cA_0(f,g;K)$. The restriction concerns only the approximating polynomials; there is no restriction on the formulas defining $f$ and $g$.

For a nonempty compact set $X\subset\C^n$ we define its $S$-polynomial hull by
\begin{equation}\label{eq:Shull}
 \widehat X^{\,S}:=\{z\in\C^n:|P(z)|\le\|P\|_X\text{ for every }P\in\cP^S\}.
\end{equation}
We call $X$ $S$-polynomially convex if $\widehat X^{\,S}=X$. Since $\cP^S$ is contained in the algebra of all polynomials,
\[
 X\subset\widehat X\subset\widehat X^{\,S},
\]
where $\widehat X$ is the ordinary polynomial hull. For the standard simplex the two hulls agree. Proposition~\ref{prop:cones} shows that the restricted hull depends only on $\Gamma$, including its boundary directions.

Unlike $\widehat X$, the set $\widehat X^{\,S}$ need not be bounded. In two variables, suppose that $\intr\Gamma\ne\emptyset$ and $\Gamma$ meets the coordinate axes only at $0$. Every nonconstant admissible monomial then vanishes on
\[
 H:=\{(Z,W)\in\C^2:ZW=0\}.
\]
Consequently $\widehat{\{0\}}^{\,S}=H$: all points of $H$ give the same polynomial values as $0$, whereas an admissible monomial with both exponents positive separates any point outside $H$ from $0$.

It is therefore useful to keep the ambient set explicit. For a set $\Omega\subset\C^n$ and a compact set $X\subset\Omega$, put
\begin{equation}\label{eq:relativeShull}
 \widehat X^{\,S}_{\Omega}:=\Omega\cap\widehat X^{\,S}.
\end{equation}
We say that $X$ is $S$-polynomially convex in $\Omega$ if $\widehat X^{\,S}_{\Omega}=X$, and put $\widehat\emptyset^{\,S}_{\Omega}:=\emptyset$. No openness assumption on $\Omega$ is needed. Taking $\Omega=\C^n$ gives the ambient notion in \eqref{eq:Shull}.

The next lemma has the same union-and-zero-fibre formulation as the classical Kallin lemma surveyed by de Paepe \cite{depaepe2001}; see also \cite[Theorem~1.1]{chi2009}. Here the separating polynomial and all hulls in the ambient space use the restricted algebra $\cP^S$.

\begin{lemma}\label{lem:kallinHull}
Let $\Omega\subset\C^n$, and let $X_1,X_2\subset\Omega$ be nonempty compact sets. Suppose that
\begin{enumerate}
\item $X_1$ and $X_2$ are $S$-polynomially convex in $\Omega$;
\item $Y_1,Y_2\subset\C$ are polynomially convex compact sets with $Y_1\cap Y_2=\{0\}$ and $0\in\partial Y_1\cap\partial Y_2$;
\item $Q\in\cP^S$ satisfies $Q(X_j)\subset Y_j$ for $j=1,2$;
\item $E:=(X_1\cup X_2)\cap Q^{-1}(0)$ is $S$-polynomially convex in $\Omega$.
\end{enumerate}
Then $X_1\cup X_2$ is $S$-polynomially convex in $\Omega$.
\end{lemma}

\begin{proof}
Put $X:=X_1\cup X_2$ and $Y:=Y_1\cup Y_2$. The ordinary planar Kallin lemma, applied to $Y_1,Y_2$ and the identity polynomial, gives $\widehat Y=Y$. Also $0\in\partial Y$: otherwise a punctured disk about $0$ would be the union of two disjoint relatively closed nonempty sets, its intersections with $Y_1\setminus\{0\}$ and $Y_2\setminus\{0\}$, contradicting connectedness. If one of these intersections were empty, $0$ would be an interior point of the other $Y_j$, again a contradiction.

Fix $z\in\widehat X^{\,S}_{\Omega}$ and put $t:=Q(z)$. For every one-variable polynomial $q$ the composition $q\circ Q$ belongs to $\cP^S$, so $t\in\widehat{Q(X)}\subset Y$. Evaluation at $z$ extends to a norm-one multiplicative functional $\Lambda_z$ on $B:=\cP^S(X)$, since $|P(z)|\le\|P\|_X$ for $P\in\cP^S$. Thus, for every $\varphi\in P(Y)$,
\begin{equation}\label{eq:functionalcalculus}
 \varphi\circ Q\in B,\qquad \Lambda_z(\varphi\circ Q)=\varphi(t).
\end{equation}
These assertions follow by uniformly approximating $\varphi$ on $Y$ by polynomials.

Suppose first that $t\in Y_j\setminus\{0\}$. Define $\varphi_j(\zeta):=\zeta$ on $Y_j$ and $\varphi_j(\zeta):=0$ on the other set. This function is continuous on $Y$ and holomorphic on its interior, so Mergelyan's theorem gives $\varphi_j\in P(Y)$. For $P\in\cP^S$ and every integer $k\ge1$, \eqref{eq:functionalcalculus} yields
\[
 |t|\,|P(z)|^k
 =|\Lambda_z((\varphi_j\circ Q)P^k)|
 \le\|\varphi_j\|_Y\,\|P\|_{X_j}^k.
\]
Taking $k$th roots and letting $k\to\infty$ gives $|P(z)|\le\|P\|_{X_j}$. Hence $z\in\widehat{X_j}^{\,S}_{\Omega}=X_j$.

Now suppose $t=0$. Every boundary point of a polynomially convex planar compact set is a peak point for its polynomial algebra; see \cite[Lemma~2.4]{chi2009}. Choose $\psi\in P(Y)$ with $\psi(0)=1$ and $|\psi|<1$ on $Y\setminus\{0\}$. For $P\in\cP^S$ and $k\ge1$,
\[
 |P(z)|=|\Lambda_z(P(\psi\circ Q)^k)|
 \le\|P(\psi\circ Q)^k\|_X.
\]
If $E\ne\emptyset$, the right-hand side tends to $\|P\|_E$: near $E$, continuity bounds $|P|$ by $\|P\|_E+\varepsilon$, while on the compact complement of that neighbourhood $|\psi\circ Q|$ has maximum strictly less than $1$. Thus $z\in\widehat E^{\,S}_{\Omega}=E$. If $E=\emptyset$, taking $P=1$ instead gives $1\le\|\psi\circ Q\|_X^k\to0$, a contradiction. This proves $\widehat X^{\,S}_{\Omega}=X$.
\end{proof}

With $\Omega=\C^n$, Lemma~\ref{lem:kallinHull} says that the union of two $S$-polynomially convex compact sets is $S$-polynomially convex under the stated separation and zero-fibre hypotheses. With the standard simplex it recovers the classical lemma. The proof adapts the classical localization argument: a piecewise planar function isolates each piece, and powers of a peak function isolate the zero fibre. It requires neither finite generation of $\Gamma\cap\Nn^n$ nor compactness of a restricted hull.

For later use we also record how uniform density determines the restricted hull. The following proposition separates the analytic approximation issue from the points that admissible polynomials cannot distinguish.

\begin{proposition}\label{prop:saturation}
Let $X\subset\C^n$ be a nonempty compact set with $\cP^S(X)=C(X)$. Then
\begin{equation}\label{eq:saturation}
 \widehat X^{\,S}=\{z\in\C^n:\text{there is }x\in X\text{ with }P(z)=P(x)\text{ for every }P\in\cP^S\}.
\end{equation}
In particular, if $\cP^S$ separates points of a set $\Omega$ containing $X$, then $X$ is $S$-polynomially convex in $\Omega$.
\end{proposition}

\begin{proof}
For $z\in\widehat X^{\,S}$, evaluation at $z$ extends to a character of $\cP^S(X)=C(X)$. Every character of $C(X)$ is evaluation at a point $x\in X$, which gives the asserted equalities. Conversely, those equalities imply all inequalities defining $\widehat X^{\,S}$. If $z,x\in\Omega$ and $\cP^S$ separates points of $\Omega$, they force $z=x$.
\end{proof}

In two variables with $\intr\Gamma\ne\emptyset$ and no nonzero axis direction, the equivalence classes defined by equality of all admissible polynomial values are the singletons in $(\C^*)^2$ and the set $H$. Indeed, choose $\alpha,\alpha+(1,0),\alpha+(0,1)\in M$ with both coordinates of $\alpha$ positive; this is possible because the open cone $\intr\Gamma$ contains a lattice point $\beta$, and $\beta+e/N\in\intr\Gamma$ for $e\in\{(1,0),(0,1)\}$ and all large $N$, so that $\alpha:=N\beta$ works. On $(\C^*)^2$, the ratios of the latter two monomials to $Z^{\alpha_1}W^{\alpha_2}$ recover $Z,W$; that denominator also separates every point of $(\C^*)^2$ from $H$. Consequently
\begin{equation}\label{eq:OmegaStar}
 \Omega_*:=(\C^*)^2\cup\{(0,0)\}
\end{equation}
is a useful ambient set on which $\cP^S$ separates points. If $\cP^S(X)=C(X)$ and $X\cap H\ne\emptyset$, then \eqref{eq:saturation} gives $\widehat X^{\,S}=X\cup H$. Thus uniform density alone does not imply ambient $S$-polynomial convexity.

\section{Weighted approximation}

We next give the weighted approximation statement used to construct two generators. In case \textup{(a)} with $d=0$ it is the approximation form of Kallin's lemma for separating polynomials with values on rays; compare \cite{kallin1966,depaepe2001,stout2007}. Its proof uses nothing but the functional calculus of the single function $Q$ inside the algebra $B$.

\begin{lemma}\label{lem:kallin}
Let $K$ be a compact Hausdorff space, let $B\subset C(K)$ be a closed subalgebra containing the constants, let $Q\in B$, and put $E:=Q^{-1}(0)$. Assume that one of the following conditions holds:
\begin{itemize}
\item[\textup{(a)}] $Q(K)\subset\bigcup_{j=1}^se^{i\theta_j}[0,\infty)$, where the numbers $e^{i\theta_j}$ are pairwise distinct; in this case put $K_j^*:=Q^{-1}\bigl(e^{i\theta_j}(0,\infty)\bigr)$;
\item[\textup{(b)}] $\opRe Q\ge0$ on $K$; in this case put $s:=1$ and $K_1^*:=K\setminus E$.
\end{itemize}
Let $u\in C(K)$ be constant on $E$, and suppose that for every $j$ and every $\varepsilon>0$ there are an integer $d\ge0$ and a function $w\in C(K)$ such that $Q^dw\in B$ and $|u-w|<\varepsilon$ on $K_j^*$. Then $u\in B$.
\end{lemma}

\begin{proof}
Let $c$ be the value of $u$ on $E$ if $E\ne\emptyset$, let $c=0$ otherwise, and put $v:=u-c$. Fix $\varepsilon\in(0,1)$. Since $v$ is continuous and vanishes on the compact set $E$, there is $\delta>0$ such that $|v|<\varepsilon$ on $K_\delta:=\{x\in K:|Q(x)|\le2\delta\}$. For each $j$ choose $d_j$ and $w_j$ as in the hypothesis and put $d:=1+\max_jd_j$ and $v_j:=w_j-c$. Since $Q$ and the constants belong to $B$, we have $Q^dv_j=Q^{d-d_j}\,Q^{d_j}w_j-cQ^d\in B$, and $|v-v_j|<\varepsilon$ on $K_j^*$.

Assume \textup{(a)}. For each $j$ choose a continuous function $\chi_j$ on $\Sigma:=\bigcup_ke^{i\theta_k}[0,\infty)$ with values in $[0,1]$, vanishing outside $e^{i\theta_j}[\delta,\infty)$ and equal to $1$ on $e^{i\theta_j}[2\delta,\infty)$, and put $\psi_j(t):=t^{-d}\chi_j(t)$ for $t\ne0$, extended by $\psi_j(0):=0$. The compact set $Q(K)\subset\Sigma$ has empty interior and connected complement, so by Mergelyan's theorem \cite{gamelin} $\psi_j$ is a uniform limit of polynomials on $Q(K)$, and therefore $\psi_j(Q)\in B$. It follows that
\[
 U:=\sum_{j=1}^s\psi_j(Q)\,Q^dv_j=\sum_{j=1}^s\chi_j(Q)\,v_j\in B.
\]
At points of $E$ both $U$ and $v$ vanish. If $x\in K_k^*$, then $\chi_j(Q(x))=0$ for $j\ne k$, and since $1-\chi_k(Q(x))\ne0$ only when $x\in K_\delta$,
\[
 |v-U|\le\chi_k(Q)\,|v-v_k|+\bigl(1-\chi_k(Q)\bigr)|v|<2\varepsilon .
\]
As $K=E\cup\bigcup_kK_k^*$, we obtain $\|v-U\|_K\le2\varepsilon$.

Assume \textup{(b)}. For $\eta>0$ the function $(t+\eta)^{-d}$ is holomorphic near the compact convex set $\{t:\opRe t\ge0,\ |t|\le\|Q\|_K\}$, which is polynomially convex, so it is a uniform limit of polynomials there and $(Q+\eta)^{-d}\in B$. Put $\rho_\eta:=Q/(Q+\eta)$; then $|\rho_\eta|\le1$ because $\opRe Q\ge0$, and
\[
 U_\eta:=(Q+\eta)^{-d}\,Q^dv_1=\rho_\eta^dv_1\in B .
\]
On $E$ both functions vanish, since $d\ge1$. On $K\setminus E$ we have $|v-U_\eta|\le|v-v_1|+|1-\rho_\eta^d|\,|v|$, which is at most $3\varepsilon$ on $K_\delta$. On $K\setminus K_\delta$ we have $|Q+\eta|\ge|Q|\ge2\delta$, hence $|1-\rho_\eta^d|\le d|1-\rho_\eta|\le d\eta/(2\delta)$, and this is smaller than $\varepsilon/(1+\|v\|_K)$ for small $\eta$. Thus $\|v-U_\eta\|_K\le3\varepsilon$ for small $\eta$.

In both cases $v$ lies within $3\varepsilon$ of $B$ for every $\varepsilon$. Since $B$ is closed and contains the constants, $u=v+c\in B$.
\end{proof}

Taking $B=\cP^S(X)$ and $d=0$ gives the following approximation consequence on finitely many rays. Its density conclusion is distinct from the hull conclusion of Lemma~\ref{lem:kallinHull}.

\begin{corollary}\label{cor:kallinS}
Let $X_1,\dots,X_s\subset\C^n$ be compact sets, $X:=\bigcup_jX_j$, and let $Q\in\cP^S$ satisfy $Q(X_j)\subset e^{i\theta_j}[0,\infty)$, where the numbers $e^{i\theta_j}$ are pairwise distinct. Put $E:=X\cap Q^{-1}(0)$. If $u\in C(X)$ is constant on $E$ and $u|_{X_j}\in\cP^S(X_j)$ for every $j$, then $u\in\cP^S(X)$. In particular, if $\cP^S(X_j)=C(X_j)$ for every $j$ and $E$ has at most one point, then $\cP^S(X)=C(X)$, and $X$ is polynomially convex.
\end{corollary}

\begin{proof}
Condition \textup{(a)} of Lemma~\ref{lem:kallin} holds for $B=\cP^S(X)$, which contains $Q$ because $\cP^S$ is an algebra. A point $x\in X$ with $Q(x)\in e^{i\theta_j}(0,\infty)$ cannot lie in $X_k$ for $k\ne j$, so the set $K_j^*$ of the lemma is contained in $X_j$. Uniform approximants of $u$ on $X_j$ by elements of $\cP^S$ provide the functions $w$ with $d=0$. For the last assertion note that $\cP^S(X)\subset P(X)$, so $P(X)=C(X)$, and the maximal ideal space of $P(X)$, which is the polynomial hull of $X$, reduces to $X$.
\end{proof}

The conclusion $\cP^S(X)=C(X)$ in Corollary~\ref{cor:kallinS} yields ordinary polynomial convexity of $X$, but not $S$-polynomial convexity in $\C^n$. By the discussion following Proposition~\ref{prop:saturation}, in two variables with no nonzero axis direction, a set $X$ with $\cP^S(X)=C(X)$ that meets $H$ has restricted hull $X\cup H$. It is, however, $S$-polynomially convex in $\Omega_*$. This is why Lemma~\ref{lem:kallinHull} is formulated for relative hulls.

In two variables the separating function will be a monomial, and the admissible pairs we use are homogeneous of the same degree. The following elementary fact describes exactly which points such a pair fails to separate.

\begin{lemma}\label{lem:homog}
Let $\beta\in\Nn^2$ with $\beta_1,\beta_2\ge1$, let $N\ge1$, and put
\[
 h:=Z^{\beta_1}W^{\beta_2},\quad F:=Z^{N\beta_1+1}W^{N\beta_2},\quad G:=Z^{N\beta_1}W^{N\beta_2+1},\quad L:=N(\beta_1+\beta_2)+1 .
\]
Then $F=h^NZ$, $G=h^NW$ and $F^{\beta_1}G^{\beta_2}=h^L$. Moreover, two points $p,p'\in(\C^*)^2$ satisfy $F(p)=F(p')$ and $G(p)=G(p')$ if and only if $p'=\omega p$ for some $\omega$ with $\omega^L=1$.
\end{lemma}

\begin{proof}
The identities are checked by comparing exponents; for instance the exponent of $Z$ in $F^{\beta_1}G^{\beta_2}$ is $\beta_1(N\beta_1+1)+\beta_2N\beta_1=\beta_1L$. If $F$ and $G$ agree at $p$ and $p'$, then $h(p')^L=h(p)^L$, so $h(p')=\zeta h(p)$ with $\zeta^L=1$, and $p'=(F,G)(p')/h(p')^N=\zeta^{-N}p$. Conversely, $F$ and $G$ are homogeneous of degree $L$.
\end{proof}

\section{Two admissible generators}

From now on $n=2$ and $a_S<b_S$. Let $\beta$ be a lattice point in $\intr\Gamma$, so that $\beta_1,\beta_2\ge1$ and $a_S<\beta_2/\beta_1<b_S$ by \eqref{eq:interior}. We call an integer $N\ge1$ \emph{admissible for $\beta$} if
\begin{equation}\label{eq:nu}
 \nu^{(1)}:=N\beta+(1,0)\in M,\qquad\nu^{(2)}:=N\beta+(0,1)\in M .
\end{equation}

\begin{lemma}\label{lem:weights}
The admissible integers for $\beta$ are exactly the integers $N\ge N_S(\beta)$, where $N_S(\beta)$ denotes the smallest one, and
\begin{equation}\label{eq:NS}
 N_S(\beta)\le1+\max\Bigl\{\frac{a_S}{\beta_2-a_S\beta_1},\ \frac{1}{b_S\beta_1-\beta_2}\Bigr\},
\end{equation}
where the second quotient is read as $0$ if $b_S=\infty$.
\end{lemma}

\begin{proof}
If $N\beta+e\in M$ then $(N+1)\beta+e\in M$, because $M$ is a semigroup containing $\beta$; hence the admissible integers form a final segment. The slope of $N\beta+(1,0)$ is $N\beta_2/(N\beta_1+1)$, which is smaller than $\beta_2/\beta_1<b_S$ and exceeds $a_S$ precisely when $N(\beta_2-a_S\beta_1)>a_S$. Similarly the slope of $N\beta+(0,1)$ exceeds $\beta_2/\beta_1>a_S$ and is smaller than $b_S$ precisely when $N(b_S\beta_1-\beta_2)>1$. If $N$ is larger than both quotients in \eqref{eq:NS}, the two vectors therefore lie in $\intr\Gamma$ by \eqref{eq:interior}, and being lattice points they belong to $M$.
\end{proof}

Boundary membership can lower the bound \eqref{eq:NS}: for the Newton polygon of $1+Z^2W+ZW^2$ and $\beta=(1,1)$ it gives $N_S(\beta)\le2$, whereas $N_S(\beta)=1$ because the boundary vectors $(2,1)$ and $(1,2)$ belong to $\Gamma$. Given $f,g\in C(K)$ and an admissible $N$ we put
\begin{equation}\label{eq:hFG}
 h:=f^{\beta_1}g^{\beta_2},\quad F:=h^Nf,\quad G:=h^Ng,\quad L:=N(\beta_1+\beta_2)+1,\quad E:=h^{-1}(0),
\end{equation}
so that $F$ and $G$ are values of the admissible monomials with exponents $\nu^{(1)},\nu^{(2)}$, and $F^{\beta_1}G^{\beta_2}=h^L$ by Lemma~\ref{lem:homog}. We consider two conditions on the compact planar set $h(K)$:
\begin{align}
 &h(K)\subset\{0\}\cup\{t\ne0:|\arg t|\le\theta\}\quad\text{for some }\theta<\pi/L; \tag{S}\label{eq:S}\\
 &h(K)\subset\textstyle\bigcup_{j=1}^se^{i\theta_j}[0,\infty),\quad\text{where the numbers }e^{iL\theta_j}\text{ are pairwise distinct}. \tag{R}\label{eq:R}
\end{align}

The next theorem identifies the functions generated by the two admissible monomials. If $a_S>0$ and $b_S<\infty$, the cone meets the coordinate axes only at the origin, as required in its last assertion.

\begin{theorem}\label{thm:core}
Let $\beta\in\intr\Gamma\cap\Nn^2$, let $N\ge N_S(\beta)$, let $K$ be a compact Hausdorff space and $f,g\in C(K)$, and define $h,F,G,L,E$ by \eqref{eq:hFG}. If \eqref{eq:S} or \eqref{eq:R} holds, then
\begin{equation}\label{eq:core}
 \overline{\C[F,G]}=J_E:=\{u\in\cA_0(f,g;K):u|_E\text{ is constant}\}.
\end{equation}
Consequently, if $E$ has at most one point, or if $f=g=0$ on $E$, then
\begin{equation}\label{eq:fullcore}
 \overline{\C[F,G]}=\cA_S(f,g;K)=\cA_0(f,g;K).
\end{equation}
If moreover $\Gamma$ meets the coordinate axes only at the origin, then, whatever the cardinality of $E$,
\begin{equation}\label{eq:collapse}
 \cA_S(f,g;K)=J_E,
\end{equation}
and in particular $\cA_S(f,g;K)=C(K)$ if and only if $\cA_0(f,g;K)=C(K)$ and $E$ has at most one point.
\end{theorem}

\begin{proof}
Let $B:=\overline{\C[F,G]}$. Since $F$ and $G$ are polynomials in $f,g$ that vanish on $E$, every element of $B$ belongs to $\cA_0(f,g;K)$ and is constant on $E$; thus $B\subset J_E$.

We first show that $h\in B$. The function $h^L=F^{\beta_1}G^{\beta_2}$ belongs to $B$. Under \eqref{eq:S}, $h^L(K)$ lies in the compact set $T:=\{0\}\cup\{t\ne0:|t|\le\|h\|^L_K,\ |\arg t|\le L\theta\}$, which has connected complement because $L\theta<\pi$. The principal branch $t\mapsto t^{1/L}$ is continuous on $T$ and holomorphic in its interior, so it is a uniform limit of polynomials on $T$ by Mergelyan's theorem, and it maps $h^L$ back to $h$ because $|\arg h|\le\theta<\pi/L$. Under \eqref{eq:R}, the map $t\mapsto t^L$ is injective on $h(K)$, since two points of distinct rays with equal $L$th powers would force $e^{iL\theta_j}=e^{iL\theta_k}$. Its inverse is therefore continuous on the compact set $h^L(K)$, which lies in a finite union of rays and hence satisfies $P(h^L(K))=C(h^L(K))$. In both cases $h\in B$.

We now apply Lemma~\ref{lem:kallin} with $Q=h$. Under \eqref{eq:S} we have $\opRe h\ge0$, since $\theta<\pi/L\le\pi/3$, so condition \textup{(b)} holds; under \eqref{eq:R} condition \textup{(a)} holds. Let $u\in J_E$ and $\varepsilon>0$, and choose a polynomial $P$ of degree $m$ with $\|u-P(f,g)\|_K<\varepsilon$. For $d\ge Nm$ every term satisfies $h^df^jg^k=h^{d-N(j+k)}F^jG^k$, so $h^dP(f,g)\in B$. The lemma, applied with $w=P(f,g)$, gives $u\in B$, which proves \eqref{eq:core}.

If $E$ has at most one point, every function is constant on $E$; if $f=g=0$ on $E$, every element of $\cA_0(f,g;K)$ takes on $E$ the value of the constant term of its approximants. In both cases $J_E=\cA_0(f,g;K)$, and \eqref{eq:fullcore} follows from $B\subset\cA_S(f,g;K)\subset\cA_0(f,g;K)$. If $\Gamma$ meets the axes only at the origin, every nonzero $(j,k)\in M$ has $j,k\ge1$; as $h=0$ forces $f=0$ or $g=0$, every nonconstant admissible monomial $f^jg^k$ vanishes on $E$, and $\cA_S(f,g;K)\subset J_E=B$. This proves \eqref{eq:collapse}. Finally, if $E$ contains two points, Urysohn's lemma provides a continuous function that is not constant on $E$, so $J_E\ne C(K)$.
\end{proof}

\begin{remark}\label{rem:meaning}
Lemma~\ref{lem:homog} explains the role of \eqref{eq:S} and \eqref{eq:R}. Off $E$, the pair $(F,G)$ identifies two points $x,y$ only if $(f,g)(y)=\omega(f,g)(x)$ with $\omega^L=1$, and then $h(y)=\omega^{\beta_1+\beta_2}h(x)$. Since $\gcd(\beta_1+\beta_2,L)=1$, the factor $\omega^{\beta_1+\beta_2}$ is a nontrivial $L$th root of unity unless $\omega=1$. Under \eqref{eq:S} two nonzero points of $h(K)$ cannot differ by such a factor, because their arguments differ by at most $2\theta<2\pi/L$, and under \eqref{eq:R} this would force two of the numbers $e^{iL\theta_j}$ to coincide. Theorem~\ref{thm:core} thus says that, once this arithmetic coincidence is excluded, the only information lost in passing from $(f,g)$ to $(F,G)$ is the collapse of the zero fibre $E$.
\end{remark}

\begin{corollary}\label{cor:quotient}
Under the hypotheses of Theorem~\ref{thm:core}, suppose that $\cA_0(f,g;K)=C(K)$. Then $\overline{\C[F,G]}$ is the algebra $C(K/E)$ of continuous functions constant on $E$, the map $(F,G)$ induces a homeomorphism of the quotient $K/E$ onto $X:=(F,G)(K)$, and $P(X)=C(X)$; in particular $X$ is polynomially convex. If $E$ is finite with $s\ge1$ points, then $\overline{\C[F,G]}$ has codimension $s-1$ in $C(K)$.
\end{corollary}

\begin{proof}
By \eqref{eq:core} the closed algebra generated by $F,G$ is $C(K/E)$, which separates distinct classes of the quotient. Since its elements are uniform limits of polynomials in $F,G$, the pair $(F,G)$ itself separates distinct classes, and it is constant on $E$; hence it induces a continuous bijection from the compact space $K/E$ onto $X$, which is a homeomorphism. Pulling back by this homeomorphism identifies $P(X)$ with $\overline{\C[F,G]}=C(K/E)$, so $P(X)=C(X)$. If $E=\{x_1,\dots,x_s\}$, the conditions $u(x_i)=u(x_1)$, $2\le i\le s$, are independent linear conditions on $C(K)$.
\end{proof}

\begin{example}[Two admissible monomials on a disk]\label{ex:simple}
Let $S$ be any compact convex set with $a_S<1<b_S$, let $D_r:=\{|z|\le r\}$, $f(z)=z$, $g(z)=\bar z$ and $\beta=(1,1)$, and let $N\ge N_S(1,1)$. Here $h=|z|^2$ satisfies \eqref{eq:S} with $\theta=0$, $E=\{0\}$, and $\cA_0(z,\bar z;D_r)=C(D_r)$ by the Stone--Weierstrass theorem. Theorem~\ref{thm:core} gives
\[
 \overline{\C\bigl[z|z|^{2N},\ \bar z|z|^{2N}\bigr]}=\cA_S(z,\bar z;D_r)=C(D_r).
\]
For the Newton polygon of $1+Z^2W+ZW^2$ one may take $N=1$; if $a_S=0.99$ and $b_S=1.01$, Lemma~\ref{lem:weights} shows that every $N\ge101$ is admissible. A narrow cone around the diagonal thus forces a high weight $|z|^{2N}$, but never destroys density.
\end{example}

The sector condition is stable under continuous relative perturbations, which gives a transfer principle requiring no differentiability.

\begin{corollary}\label{cor:perturb}
Let $a,b,c,d\in\Nn$ with $p:=a-b>0$ and $q:=d-c>0$, put $d_0:=\gcd(p,q)$, suppose that $a_S<p/q<b_S$, and put $\beta:=(q/d_0,p/d_0)$, which lies in $\intr\Gamma$ by \eqref{eq:interior}. Let $f,g$ be continuous near $0$ with $f(0)=g(0)=0$ and
\begin{equation}\label{eq:relerr}
 f(z)=z^a\bar z^{\,b}\bigl(1+\varepsilon_f(z)\bigr),\qquad
 g(z)=z^c\bar z^{\,d}\bigl(1+\varepsilon_g(z)\bigr),\qquad
 \varepsilon_f(z),\varepsilon_g(z)\to0\ \ (z\to0).
\end{equation}
Then for every $N\ge N_S(\beta)$ there is $r_0=r_0(N)>0$ such that, for $0<r\le r_0$,
\[
 \cA_S(f,g;D_r)=\cA_0(f,g;D_r)=\overline{\C[h^Nf,\,h^Ng]}^{\,D_r},\qquad h=f^{q/d_0}g^{p/d_0}.
\]
\end{corollary}

\begin{proof}
Since $p\beta_1-q\beta_2=0$, the leading term of $h$ has no angular dependence:
\[
 h(z)=|z|^{\beta_1(a+b)+\beta_2(c+d)}\bigl(1+\varepsilon_f(z)\bigr)^{\beta_1}\bigl(1+\varepsilon_g(z)\bigr)^{\beta_2}\qquad(z\ne0).
\]
For small $r$ the last two factors are nonzero on $D_r\setminus\{0\}$ and their product has argument smaller than $\pi/L$ in modulus. Hence $E=\{0\}$ and \eqref{eq:S} holds on $D_r$, and \eqref{eq:fullcore} gives the claim.
\end{proof}

For exact mixed monomials the density question has a complete answer, which combines the arithmetic condition of Minsker \cite{minsker1976} with the position of the balance slope $p/q$ relative to the two extreme slopes of $S$. For the standard simplex and the pure monomials $f_0=z^p$, $g_0=\bar z^{\,q}$, it reduces to Minsker's theorem.

\begin{theorem}\label{thm:sharp}
Let $f_0(z)=z^a\bar z^{\,b}$ and $g_0(z)=z^c\bar z^{\,d}$ with $p:=a-b>0$ and $q:=d-c>0$. For every $r>0$,
\[
 \cA_S(f_0,g_0;D_r)=C(D_r)\quad\text{if and only if}\quad\gcd(p,q)=1\ \text{ and }\ a_S<\frac pq<b_S .
\]
When these conditions hold, two explicit admissible monomials already generate $C(D_r)$.
\end{theorem}

\begin{proof}
By \eqref{eq:interior} the condition on slopes means $(q,p)\in\intr\Gamma$ (when $a_S=b_S$ it fails, and $\intr\Gamma=\emptyset$). On the circle $|z|=\rho>0$ the function $f_0^jg_0^k$ is a constant multiple of $e^{i(pj-qk)\phi}$, $z=\rho e^{i\phi}$. If $d_0:=\gcd(p,q)>1$, all these frequencies are multiples of $d_0$, so the Fourier coefficient of index $1$ of every element of $\cA_S$ vanishes on the circle and $z$ cannot be approximated. Suppose $(q,p)\notin\intr\Gamma$. If $\Gamma$ contained vectors $v,w$ with $\ell(v)>0>\ell(w)$, where $\ell(j,k):=pj-qk$, then the open cone spanned by $v$ and $w$ would be an open subset of $\Gamma$ meeting the kernel of $\ell$ in the quadrant, that is, the ray through $(q,p)$; thus $\ell$ has a constant weak sign on $\Gamma$, and either $z$ or $\bar z$ has a Fourier coefficient that all elements of $\cA_S$ lack. This proves necessity.

Conversely, let $\gcd(p,q)=1$ and $(q,p)\in\intr\Gamma$. By B\'ezout's identity there are integer vectors with $p\nu_1-q\nu_2=\pm1$, and adding large multiples of $(q,p)$ we obtain $\nu^+,\nu^-\in M$ with $p\nu^+_1-q\nu^+_2=1$ and $p\nu^-_1-q\nu^-_2=-1$. Then $F_0:=f_0^{\nu^+_1}g_0^{\nu^+_2}=z|z|^{2s}$ and $G_0:=f_0^{\nu^-_1}g_0^{\nu^-_2}=\bar z|z|^{2t}$ for some integers $s,t\ge0$, and $H_0:=F_0G_0=|z|^{2(s+t+1)}$. By the Weierstrass theorem on the interval $H_0(D_r)$, the functions
\[
 F_0\,(H_0+\eta)^{-s/(s+t+1)},\qquad G_0\,(H_0+\eta)^{-t/(s+t+1)}
\]
belong to $\overline{\C[F_0,G_0]}$; they are bounded by $|z|$ and converge to $z$ and $\bar z$ uniformly away from $0$, hence uniformly on $D_r$ as $\eta\downarrow0$. The Stone--Weierstrass theorem gives $\overline{\C[F_0,G_0]}=C(D_r)$.
\end{proof}

\begin{example}[Two independent obstructions]\label{ex:obstruct}
For $f_0=z$, $g_0=\bar z$ and any $S$ with $b_S=1$, for instance the triangle $\operatorname{conv}\{0,(1,0),(1,1)\}$, the unrestricted algebra is $C(D_r)$, but every admissible exponent satisfies $j-k\ge0$, so no admissible polynomial approximates $\bar z$ on $|z|=r$; here the balance slope $1$ equals $b_S$. For $f_0=z^2$, $g_0=\bar z^2$ and any $S$ with $a_S<1<b_S$, the balance slope lies strictly between the extreme slopes, but both generators identify $z$ and $-z$. Thus the two conditions of Theorem~\ref{thm:sharp} are independent.
\end{example}

\section{Examples}\label{sec:examples}

We now illustrate each mechanism of the previous sections. Except in the first example, $S$ is an arbitrary compact convex set subject only to conditions on $a_S$ and $b_S$, and whenever the diagonal direction is involved we write $N_S:=N_S(1,1)$ for the least admissible weight, which is controlled by \eqref{eq:NS}.

\begin{example}[An irrational cone and a perturbation]\label{ex:irrational}
Let $\Gamma=\{(x,y)\in\R_+^2:\tfrac{\sqrt2}{3}x\le y\le\tfrac{\sqrt5}{2}x\}$ and $S=\Gamma\cap\{x+y\le1\}$. The semigroup $M$ is not finitely generated: otherwise the cone spanned by $M$ would be closed and rational polyhedral, and since it contains every rational direction of the closed cone $\Gamma$ it would coincide with $\Gamma$, whose boundary slopes are irrational. Let
\[
 f(z)=z^2,\qquad g(z)=\bar z^3+\lambda z^4,\qquad\lambda\in\C .
\]
Then $p=2$, $q=3$, and $\beta=(3,2)$ has slope $2/3\in(a_S,b_S)=(\sqrt2/3,\sqrt5/2)$. The two quotients in \eqref{eq:NS} are approximately $0.80$ and $0.74$, so $N=1$ is admissible, giving the exponent vectors $(4,2)$ and $(3,3)$. Since $g-\lambda f^2=\bar z^3$, the unrestricted algebra contains $z^2$ and $\bar z^3$, hence also
\[
 H_0=(z^2)^3(\bar z^3)^2=|z|^{12},\qquad F_0=(z^2)^2\bar z^3=z|z|^6,\qquad G_0=z^2\bar z^3=\bar z|z|^4,
\]
and the argument of Theorem~\ref{thm:sharp} with the exponents $1/2$ and $1/3$ shows that $\cA_0(f,g;D_r)=C(D_r)$. The relative error of $g$ has modulus $|\lambda||z|$, so Corollary~\ref{cor:perturb} gives, for all sufficiently small $r$,
\[
 \cA_S(f,g;D_r)=\overline{\C[f^4g^2,\ f^3g^3]}=C(D_r).
\]
Only the radius depends on $\lambda$.
\end{example}

\begin{example}[A zero fibre with several points]\label{ex:collapse}
Let $S$ satisfy $0<a_S<1<b_S<\infty$, let $N\ge N_S$, $0<r<1$, $k\ge1$, and let $K$ be the disjoint union of $D_r$ and $k$ isolated points $x_1,\dots,x_k$. Put $f(z)=z$ and $g(z)=\bar z$ on $D_r$, and $f(x_i)=0$, $g(x_i)=i$. The compact set $g(K)=D_r\cup\{1,\dots,k\}$ has connected complement, and a function equal to $1$ near $i$ and to $0$ near the other components is holomorphic near $g(K)$; by Runge's theorem the indicator of $\{x_i\}$ belongs to $\cA_0(f,g;K)$. Multiplying $f$ and $g$ by the indicator of $D_r$ we obtain $z$ and $\bar z$ on the disk, and the Stone--Weierstrass theorem gives $\cA_0(f,g;K)=C(K)$. On the other hand $h=fg$ equals $|z|^2$ on $D_r$ and $0$ at each $x_i$, so \eqref{eq:S} holds and $E=\{0,x_1,\dots,x_k\}$. Since $\Gamma$ meets the axes only at the origin, Theorem~\ref{thm:core} and Corollary~\ref{cor:quotient} give
\[
 \cA_S(f,g;K)=\overline{\C[f^{N+1}g^N,\ f^Ng^{N+1}]}=\{u\in C(K):u(0)=u(x_1)=\dots=u(x_k)\},
\]
a subalgebra of codimension $k$, although the unrestricted algebra is all of $C(K)$. The condition excluding the coordinate axes cannot be omitted. Suppose instead that $a_S<1$ and $(0,1)\in\Gamma$, so that $b_S=\infty$. Then $g$ itself is admissible, and $g$ takes the distinct values $0,1,\dots,k$ on $E$. By Lagrange interpolation, every $u\in C(K)$ differs from some polynomial in $g$ by a function vanishing on $E$. That function belongs to $\overline{\C[f^{N+1}g^N,f^Ng^{N+1}]}$ by \eqref{eq:core}, since $\cA_0(f,g;K)=C(K)$. Hence $\cA_S(f,g;K)=C(K)$.
\end{example}

\begin{example}[A connected compact set]\label{ex:connected}
The phenomenon does not require isolated points. Let $S$ satisfy $0<a_S<1<b_S<\infty$, let $N\ge N_S$, and let $K=[0,1]$, $f(x)=x(x-1)$, $g(x)=x$. Then $\cA_0(f,g;K)=C([0,1])$ by the Weierstrass theorem, while $h=fg=x^2(x-1)\le0$ satisfies \eqref{eq:R} with a single ray and $E=\{0,1\}$. Theorem~\ref{thm:core} gives
\[
 \cA_S(f,g;[0,1])=\overline{\C[f^{N+1}g^N,\ f^Ng^{N+1}]}=\{u\in C([0,1]):u(0)=u(1)\}.
\]
Accordingly the map $x\mapsto(F,G)(x)=\bigl(x^{2N+1}(x-1)^{N+1},\ x^{2N+1}(x-1)^N\bigr)$ identifies the endpoints and is injective on $(0,1)$, since $F/G=x-1$ there; its image is a simple closed curve in $\R^2\subset\C^2$, on which polynomials are dense in the continuous functions, in agreement with Corollary~\ref{cor:quotient}.
\end{example}

The next result is a genuine application of Kallin's lemma, both in its classical and in its restricted form, and it shows that the choice of the two admissible monomials is governed by arithmetic.

\begin{theorem}\label{thm:stars}
Let $m\ge2$, $\omega:=e^{2\pi i/m}$, $r>0$, and
\[
 X^{(m)}:=\bigcup_{j=0}^{m-1}X_j,\qquad X_j:=\{(z,\omega^j\bar z):|z|\le r\}.
\]
Let $S$ satisfy $a_S<1<b_S$.
\begin{itemize}
\item[\textup{(i)}] $\cP^S(X^{(m)})=C(X^{(m)})$.
\item[\textup{(ii)}] If $N\ge N_S$, $F_N:=Z^{N+1}W^N$ and $G_N:=Z^NW^{N+1}$, then
\[
 \overline{\C[F_N,G_N]}=C(X^{(m)})\quad\text{if and only if}\quad\gcd(2N+1,m)=1 .
\]
\end{itemize}
\end{theorem}

\begin{proof}
(i) The polynomial $Q=ZW$ belongs to $\cP^S$, maps $X_j$ onto the ray segment $\omega^j[0,r^2]$, and vanishes on $X^{(m)}$ only at the origin. The homeomorphism $z\mapsto(z,\omega^j\bar z)$ of $D_r$ onto $X_j$ pulls $Z^aW^b$ back to $\omega^{jb}z^a\bar z^{\,b}$, so $\cP^S(X_j)$ corresponds to $\cA_S(z,\bar z;D_r)$, which equals $C(D_r)$ by Theorem~\ref{thm:sharp} with $p=q=1$. Corollary~\ref{cor:kallinS} gives the claim.

(ii) Here $\beta=(1,1)\in\intr\Gamma$, $h=ZW$ and $L=2N+1$. If $\gcd(L,m)=1$, the numbers $\omega^{jL}$, $0\le j<m$, are distinct, so \eqref{eq:R} holds for $h$ on $X^{(m)}$, with $E=\{(0,0)\}$. Moreover $P(X^{(m)})=C(X^{(m)})$ by (i) applied to the standard simplex. Theorem~\ref{thm:core}, applied to the coordinate functions on $K=X^{(m)}$, gives density. If $\delta:=\gcd(L,m)>1$, put $\eta:=e^{2\pi i/\delta}$ and take $x=(z,\bar z)\in X_0$ with $z\ne0$. Then $\eta x=(\eta z,\eta^2\,\overline{\eta z})$, and $\eta^2$ is an $m$th root of unity because $\delta\mid m$, so $\eta x\in X^{(m)}$; moreover $\eta x\ne x$, while $\eta^L=1$ because $\delta\mid L$. By Lemma~\ref{lem:homog} the functions $F_N,G_N$ do not separate $x$ and $\eta x$.
\end{proof}

The star also gives an explicit hull computation. Suppose in addition that $\Gamma$ meets the coordinate axes only at $0$. By Proposition~\ref{prop:saturation} and part \textup{(i)},
\begin{equation}\label{eq:starHull}
 \widehat{X^{(m)}}^{\,S}=X^{(m)}\cup H,
 \qquad \widehat{X^{(m)}}^{\,S}_{\Omega_*}=X^{(m)}.
\end{equation}
The relative convexity assertion also follows directly from Lemma~\ref{lem:kallinHull}, applied repeatedly. Each $X_j$ is $S$-polynomially convex in $\Omega_*$ by the one-disc density result and Proposition~\ref{prop:saturation}. The polynomial $Q=ZW$ maps each piece onto a distinct ray segment; a union of these segments is polynomially convex and meets the next segment only at $0$. The zero fibre is $\{(0,0)\}$, which is $S$-polynomially convex in $\Omega_*$ because its ambient hull is $H$. This verifies every hypothesis of the union lemma at each step.

Thus, for a fixed star, the set of good weights $N\ge N_S$ is determined by the residue of $N$ modulo $m$, and it is never empty, since every multiple of $m$ is good. For $m=2$ the set $X^{(2)}$ is the union of the totally real discs $\{(z,\bar z)\}$ and $\{(z,-\bar z)\}$; since $2N+1$ is odd, every admissible pair is dense, although $h=ZW$ takes both signs and the sector condition fails. For $m=3$ the weight $N$ is good if and only if $N\not\equiv1\pmod 3$. When $N_S=1$, as for the Newton polygon of $1+Z^2W+ZW^2$, the first natural choice therefore fails: on $X_j$, with $w=\omega^jz$, one finds $Z^2W=|w|^2w$ and $ZW^2=|w|^2\bar w$ independently of $j$, whereas the weight $N=2$ gives
\[
 \overline{\C[Z^3W^2,\ Z^2W^3]}=\cP^S(X^{(3)})=C(X^{(3)}).
\]

The transfer principle also shows that the restricted problem is no easier than the classical one when the leading exponents have a common divisor. The two ordinary approximation results used below are quoted, not reproved.

\begin{theorem}\label{thm:symmetry}
Let $S$ satisfy $a_S<1<b_S$, let $N\ge N_S$, $f(z)=z^2$, and
\[
 g_+(z):=\bar z^2+z^3,\qquad g_-(z):=\bar z^2+\bar z^3 .
\]
For all sufficiently small $r>0$,
\begin{align}
 \cA_S(f,g_+;D_r)&=\overline{\C[f^{N+1}g_+^N,\ f^Ng_+^{N+1}]}=C(D_r),\label{eq:positive}\\
 \cA_S(f,g_-;D_r)&=\overline{\C[f^{N+1}g_-^N,\ f^Ng_-^{N+1}]}\ne C(D_r).\label{eq:negative}
\end{align}
\end{theorem}

\begin{proof}
De Paepe proved that $\overline{\C[z^2,\bar z^2+z^3]}=C(D_r)$ for small $r$ \cite{depaepe1996}; see also \cite[Example~3.5]{dpw2007}. O'Farrell and Sanabria-Garc\'ia proved that the disc $X_r=\{(z^2,\bar z^2+\bar z^3):|z|\le r\}$ has a nontrivial polynomial hull for small $r$ \cite{os2002}; since $\cA_0(z^2,g_-;D_r)$ is isometric to $P(X_r)$, equality with $C(D_r)$ would make $(z^2,g_-)$ separate the points of $D_r$, hence a homeomorphism of $D_r$ onto $X_r$, and would therefore force $P(X_r)=C(X_r)$ and polynomial convexity of $X_r$, so $\cA_0(z^2,g_-;D_r)\ne C(D_r)$. In both cases the balance vector is $\beta=(1,1)\in\intr\Gamma$, and $L=2N+1$. The anchors
\[
 h_+=fg_+=|z|^4\bigl(1+z^5|z|^{-4}\bigr),\qquad h_-=fg_-=|z|^4(1+\bar z)
\]
vanish only at $0$ on small disks, and their arguments tend to $0$ uniformly as $z\to0$; thus \eqref{eq:S} holds with $\theta<\pi/(2N+1)$ on small disks. Theorem~\ref{thm:core} identifies each restricted algebra with the corresponding unrestricted one, which gives \eqref{eq:positive} and \eqref{eq:negative}.
\end{proof}

Corollary~\ref{cor:perturb} transfers density but does not create it, and the following example shows that relative smallness of a perturbation alone does not preserve density, even for $z$ and $\bar z$.

\begin{example}[Relative smallness without density]\label{ex:noninjective}
For $j\ge3$ let $r_j=2^{-j}$, $s_j=r_j+r_j^2$ and $\rho_j=r_j^2/4$, and choose continuous functions $0\le\chi_j\le1$ supported in the disks $\{|z-s_j|\le\rho_j\}$ with $\chi_j(s_j)=1$; these disks are pairwise disjoint and do not contain any of the points $r_k$. Put
\[
 f(z):=z-\sum_{j\ge3}r_j^2\chi_j(z),\qquad g(z):=\bar z-\sum_{j\ge3}r_j^2\chi_j(z).
\]
The sums are locally finite off the origin, $f$ and $g$ extend continuously by $f(0)=g(0)=0$, and since $|z|\asymp r_j$ on the support of $\chi_j$, they satisfy \eqref{eq:relerr} with $(a,b,c,d)=(1,0,0,1)$. Nevertheless $f(s_j)=f(r_j)=r_j$ and $g(s_j)=g(r_j)=r_j$, so $f,g$ fail to separate points in every disk centred at $0$. For every $S$ with $a_S<1<b_S$, Corollary~\ref{cor:perturb} gives $\cA_S(f,g;D_r)=\cA_0(f,g;D_r)$ for small $r$, but neither algebra equals $C(D_r)$.
\end{example}

\begin{remark}
Under the hypotheses of Corollary~\ref{cor:perturb}, the restricted local problem reduces to the classical one for the same pair. Theorem~\ref{thm:symmetry} shows that the cone and the leading monomials alone do not determine density. Further conditions on the higher order terms, beyond the classes in \cite{dieudepaepe2002,dieuchi2006,dpw2007}, would transfer to the restricted setting by that corollary.
\end{remark}
\par
{\bf Acknowledgments.} We would to thank to Professor Nguyen Quang Dieu for his  suggestions the problem and to completely written the manuscripts.
\end{document}